\documentclass{cmslatex}
\usepackage{latexsym, amssymb, enumerate, amsmath}
\usepackage{enumerate}
\usepackage{hyperref}

\renewcommand{\theequation}{\arabic{section}.\arabic{equation}}

\def\sign{\mathop{\rm {sign}}}

\def\eps{\varepsilon}

\def\p{\partial}
\def\a{\alpha}
\def\b{\beta}
\def\S{\Sigma}

\begin{document}
\title{A simple derivation of BV bounds for inhomogeneous  relaxation systems
\thanks{{Received date / Revised version date}
}}
\author{
Beno\^{i}t Perthame
\thanks{Sorbonne Universit\'es, UPMC Univ Paris 06, CNRS UMR 7598, Laboratoire Jacques-Louis Lions, F-75005, Paris, 
\newline
CNRS, UMR 7598, Laboratoire Jacques-Louis Lions, F-75005, Paris, France, 
\newline
INRIA Paris-Roquencourt, EPC Mamba (benoit.perthame@upmc.fr).
}
\and Nicolas Seguin \thanks{Sorbonne Universit\'es, UPMC Univ Paris 06, CNRS UMR 7598, Laboratoire Jacques-Louis Lions, F-75005, Paris,
\newline
CNRS, UMR 7598, Laboratoire Jacques-Louis Lions, F-75005, Paris, (nicolas.seguin@upmc.fr).}
 \and  Magali Tournus
 \thanks{Sorbonne Universit\'es, UPMC Univ Paris 06, CNRS UMR 7598, Laboratoire Jacques-Louis Lions, F-75005, Paris
 \newline
 CNRS, UMR 7598, Laboratoire Jacques-Louis Lions, F-75005, Paris, France
 \newline
 Present address : Department of Mathematics, Pennsylvania State University,
University Park, Pennsylvania 16802, USA, (magali.tournus@ann.jussieu.fr).}
}




\pagestyle{myheadings} \markboth{A simple detour for BV bounds in relaxations systems}
{Beno\^{i}t Perthame, Nicolas Seguin, Magali Tournus}
\maketitle

\begin{abstract}
We consider relaxation systems of transport equations with heterogeneous source terms and with boundary conditions, which limits are scalar conservation laws. Classical bounds  fail in this context and in particular BV estimates. They are the most standard and simplest way to prove compactness and convergence. 

We provide a novel and simple method to obtain partial BV regularity and strong compactness
in this framework. The standard notion of entropy is not convenient either and we also indicate another, but closely related, notion. We give two examples motivated by renal flows which  consist of 2 by 2 and 3 by 3 relaxation systems with  2-velocities but the method is more general.
\end{abstract}

\begin{keywords} 
Hyperbolic relaxation; spatial heterogeneity; entropy condition; boundary conditions; strong compactness.

{\bf Subject classifications.} 35L03, 35L60, 35B40, 35Q92
\end{keywords}

\section{Introduction}
\label{sec:intro}

The usual framework of hyperbolic relaxation \cite{CLL,NAT98,Bouchut} concerns the convergence of a general hyperbolic system with stiff source terms toward 
a conservation law, when the relaxation parameter $\epsilon$ goes to zero.
More specifically, Jin and Xin \cite{Jin_Xin} introduced a $2 \times 2$ linear hyperbolic system with stiff source term that approximates any given conservation law. 
The problem of interest is to prove the convergence of the microscopic quantities depending on $\epsilon$ toward the macroscopic quantities. 
A problem entering the framework of hyperbolic relaxation is motivated by very simplified models of kidney physiology \cite{TESP, TSPTE,These_tournus} and fits the Jin and Xin framework with two major differences, boundary conditions, and spatial dependence of the source term. The type of boundary condition and the spatial dependence constitute the main novelty of the present study.

The system represents two solute concentrations $u_{\epsilon}(x,t)$ and $v_{\epsilon}(x,t)$ and  is written, for $t\geq 0$ and $x \in [0,L]$,
\begin{equation}
\label{a_evo2tubes}
\left\{
\begin{aligned}
&\dfrac{\partial u_{\epsilon}}{\partial t}(x,t)+\dfrac{\partial u_{\epsilon}}{\partial x}(x,t)=\dfrac{1}{\epsilon}\Big[ h(v_{\epsilon}(x,t),x)) - u_{\epsilon}(x,t) \Big],
\\
&\dfrac{\partial v_{\epsilon}}{\partial t }(x,t)-\dfrac{\partial v_{\epsilon}}{\partial x}(x,t)=\dfrac{1}{\epsilon}\Big[ u_{\epsilon}(x,t)-  h(v_{\epsilon}(x,t),x))  \Big],
\\[1mm]
&u_{\epsilon}(0,t)=u_0, \qquad v_{\epsilon}(L,t)=\a u_{\epsilon}(L,t),  \qquad \a \in (0,1),
\\[1mm]
&u_{\epsilon}(x,0) = u^0(x) >0, \qquad v_{\epsilon}(x,0) = v^0(x) >0.
\end{aligned}
\right.
\end{equation}
The question of interest here is to understand the behavior of $u_{\epsilon}$ and $v_{\epsilon}$ when $\epsilon$ vanishes. Another question, addressed in \cite{T}, is to explain the urine concentration mechanism due to the `pumps', on the cell membranes, represented by the nonlinearity $h(v,x)$. 

We make the following hypotheses 
\begin{equation}
 h(0,x)=0, \qquad 1< \beta \leq \dfrac{\p h}{\p v}(v, x) \leq \mu,  
\label{as_h1}
\end{equation}
\begin{equation}
\sup_{v} \displaystyle\int_0^L |\dfrac{\p h}{\p x}(v,x)| dx \leq C ,\qquad \mbox{$h(.,x)$ is not locally affine}.
\label{a_has}
\end{equation}
Here $C$, $\beta$ and $\mu$ are positive constants. Also, we use the following bounds on the initial data 
\begin{equation}
 u^0, \; v^0 \in L^\infty(0,L), \qquad \frac{d}{dx}u^0\in L^1(0,L), \qquad \frac{d}{dx}v^0 \in L^1(0,L),
\label{BV}
\end{equation}
and, the well-prepared initial data condition
\begin{equation}
 u^0(x)=h(v^0(x),x), \qquad x\in[0,L].
\label{WP}
\end{equation}

It is simple and standard to find the formal limit as $\epsilon$ vanishes. Adding the two equations of \eqref{a_evo2tubes}, we find the conservation law 
\begin{equation}
\dfrac{\partial [u_{\epsilon}+ v_{\epsilon}](x,t)}{\partial t}+\dfrac{\partial [u_{\epsilon}- v_{\epsilon}](x,t)}{\partial x} =0.
\label{CLeps}
\end{equation}
We expect that, in the limit, the idendity holds
\begin{equation*}
u(x,t)= h(v(x,t),x)
\end{equation*}
and this  identifies  the limiting quasi-linear conservation law 
\begin{equation}
\dfrac{\partial [h(v(x,t),x)+ v(x,t)]}{\partial t}+\dfrac{\partial [h(v(x,t),x) - v(x,t)]}{\partial x} =0,
\label{CLlim}
\end{equation}
and the boundary conditions on $\rho(x,t)= h(v(x,t),x)+ v(x,t)$, namely $\rho(0,t) = u_0 + h^{-1}(u_0,0)$ and $\rho(L,t)$ is free, can be identified via the standard argument of Bardos-Le Roux and N\'ed\'elec, \cite{BLN}, using the adapted Kru$\check{\mbox{z}}$kov entropies that we introduce in Section~\ref{sec:entrop}.

The justification poses particular difficulties due to the boundary condition and $x$-dependent flux.

One difficulty is that, in order to justify the above limit, an entropy identity  is needed for two reasons. Firstly  to prove that equilibrium is reached \cite{Tzavaras_relative}. Secondly, because at the limit $\epsilon =0$, a quasi-linear equation arises, shocks can be produced. Therefore, an entropy formulation is needed to define uniquely the solutions.  When applied to this conservation equation, usual convex entropies as both in \cite{Dafermos} or Kru\~{z}kov \cite{kruz}, contain a term with a derivative of $h$ with respect to $x$. 
When the regularity of $h$ is limited to BV, $h_x$ is a measure and the weak entropy formulation is not well defined. One of our goal is to present a convenient notion of entropy with $x$-dependent relaxation in order to  get rid of this problem. We use an entropy formulation introduced in \cite{AUD} and adapted to scalar conservation laws with spatial heterogeneities.

Another difficulty is that, in order to prove the validity of the limit, strong compactness is also needed. Several methods have been developed in this goal. Uniform bounds in the space of functions with bounded variation (BV in short) is the strongest method and the most standard. Such bounds are proved in  \cite{Nat_Ter}  for a similar system with boundary conditions in the homogeneous case; the result is extended with a source term in  \cite{ShuYi}. More elaborate tools are the compensated compactness method \cite{MUR} or the kinetic formulation and averaging lemmas \cite{GLPS,PER2} (in particular extending  the measure valued solutions of DiPerna \cite{dipernaMVS}). These methods  are weaker, because they give only convergence but no bounds, and thus apply to more general situations than BV bounds.  Because of  the spatial $x$-dependence of $h$,  the BV framework is not available as today for  \eqref{a_evo2tubes}. Whereas time BV estimates follow immediately from the equations, compactness in the spatial 
direction cannot be obtained in this way.  We propose a new method to prove spatial compactness. It does not use BV bounds on each component, but gives BV bounds in $x$ for a single quantity and can be applied to 
spatially heterogeneous systems. 
We are going to prove that

\begin{proposition}
We make  assumptions \eqref{as_h1}, \eqref{a_has}, \eqref{BV}, \eqref{WP} and fix a time $T$. We have
\\
(i)   $u_{\epsilon}$ and $v_{\epsilon}$ are uniformly bounded in $L^\infty ([0,L]\times[0,T])$,
\\
(ii)  $\frac{\partial}{\partial t}u_{\epsilon}$, $\frac{\partial}{\partial t}v_{\epsilon}$ and $\frac{\partial}{\partial x} [u_{\epsilon} -v_{\epsilon}]$ are bounded in $L^\infty ([0,T]; L^1(0,L))$,
\\
(iii)  there exists $v \in L^{\infty} ([0,L]\times[0,T])$ such that
\begin{equation}
 u_{\epsilon}(x,t) \underset{\epsilon \rightarrow 0}{\longrightarrow}h(v(x,t),x), \qquad v_{\epsilon}(x,t) \underset{\epsilon \rightarrow 0}{\longrightarrow}v(x,t), \; a.e.  
\end{equation}
(iv) the equation \eqref{CLlim} is satisfied and entropy inequalities hold, see section~\ref{sec:entrop}. 
\label{propo}
\end{proposition}

However it is not correct that  $\frac{\partial}{\partial x}u_{\epsilon}$ and  $\frac{\partial}{\partial x}v_{\epsilon}$ are separately bounded in $L^\infty ([0,T]; L^1(0,L))$.

\section{$L^\infty$ bound}
\label{sec:Linfty}

We first prove uniform estimates. Unlike the homogeneous case, $L^\infty$ bounds are not always true and the most general existence theory is in $L^1$, see \cite{ALD}. Here we build particular sub and supersolutions of  \eqref{a_evo2tubes} which are uniformly bounded in $\epsilon$.

\begin{lemma}
 The solution of \eqref{a_evo2tubes} satisfies the uniform estimate
\begin{equation}
 \| u_{\epsilon}\|_{L^{\infty}([0,L]\times[0,T])}  \leq K(\beta,u_0,u^0,v^0), \qquad
 \| v_{\epsilon}\|_{L^{\infty}([0,L]\times[0,T])}  \leq K(\beta,u_0,u^0,v^0) .
\label{evolest2}
\end{equation}
\label{lm:nound}
\end{lemma}

\begin{proof}
To obtain an $L^{\infty}$ bound on the time-dependent solution, we follow the approach of \cite{AUD} and use the comparison principle with appropriate supersolution.
Indeed, because of the $x-$dependence of $h$, constant
functions are not super-solution of the stationary problem.
We introduce the stationary version of \eqref{a_evo2tubes}

\begin{equation}
\label{P1C3_2tubes}
\left\{
\begin{aligned}
\dfrac{d U_{\epsilon}}{d x}(x)&=\dfrac{1}{\epsilon}\Big[ h(V_{\epsilon}(x),x)) - U_{\epsilon}(x) \Big],
\\
-\dfrac{d V_{\epsilon}}{d x}(x)&=\dfrac{1}{\epsilon}\Big[ U_{\epsilon}(x)-  h(V_{\epsilon}(x),x))  \Big],
\\[3mm]
U_{\epsilon}(0)&=U_0 >0, \qquad V_{\epsilon}(L)=\a U_{\epsilon}(L).
\end{aligned}
\right.
\end{equation}

We are going to prove that there exists a smooth super solution $(U_{\epsilon},V_{\epsilon})$ of the stationary problem \eqref{P1C3_2tubes}, and a constant $K(U_0,\beta)>0$ such that
\begin{equation}
0 \leq U_{\epsilon} (x) \leq K(U_0,\beta) , \qquad  0 \leq V_{\epsilon} \leq K(U_0,\beta).
 \label{a_supersol}
\end{equation}
This concludes the proof of Lemma~\ref{lm:nound} because   a solution of \eqref{P1C3_2tubes} where $U_0\geq u_0$, is a super-solution of \eqref{a_evo2tubes}, and the comparison principle gives $0\leq u_{\epsilon} \leq U_{\epsilon}$ and $0 \leq v_{\epsilon} \leq V_{\epsilon}$. 

Because it has  been proved in \cite{TESP} that \eqref{P1C3_2tubes} admits a unique solution 
 which lies in 
BV$([0,L]\times[0,T])$ (fixed point argument for the existence and contraction for the uniqueness), it remains to  prove that a solution of \eqref{P1C3_2tubes} with $U_0 \geq u_0$ is uniformly bounded in $\epsilon$.

Adding the two lines of \eqref{P1C3_2tubes}, we obtain a quantity which does not depend on $x$,
\begin{equation}
\label{CM}
 U_{\epsilon}(x)-V_{\epsilon}(x) =: K_{\epsilon}.
 \end{equation}
Using the boundary values, we find uniform bounds on $K_{\epsilon}$
\begin{equation}
K_{\epsilon}= U_0-V_{\epsilon}(0) \leq U_0 \qquad K_{\epsilon}=U_{\epsilon}(L) -V_{\epsilon}(L) = (1- \a) U_{\epsilon}(L) \geq 0,
\label{eq1}
\end{equation}
and thus
\begin{equation}
 0\leq K_{\epsilon} \leq U_0, \qquad U_{\epsilon}(L) \leq \dfrac{U_0}{1-\a}.
\label{eq2}
\end{equation}
Hence, we just have to prove that $U_{\epsilon}$ is uniformly bounded in $L^{\infty}$, knowing that $U_0$ and $U_{\epsilon}(L)$ are uniformly bounded in $\mathbb{R}$.
For that, we use the maximum principle assuming $\mathcal{C}^1$ regularity (one can easily  justify it for a regularized function $h(v,x)$ and pass to the limit).
Indeed, if $U_{\epsilon}$ reaches its maximal value on the boundary, the result follows from \eqref{eq2}. If $U_{\epsilon}$ reaches its maximal value at $x_0 \in ]0,L[$, then,
\begin{equation}
 0=\dfrac{dU_{\epsilon}}{d x}(x_0)=\dfrac{1}{\epsilon}\Big(h(U_{\epsilon}(x_0)-K_{\epsilon},x_0)-U_{\epsilon}(x_0)\Big),
\end{equation}
and thus by assumption \eqref{as_h1}, 
\begin{equation}
U_{\epsilon}(x_0) = h(U_{\epsilon}(x_0)-K_{\epsilon},x_0) \geq \b U_{\epsilon}(x_0)- \b K_{\epsilon}.
\end{equation}
Finally, \eqref{a_supersol} is proved because the above inequality gives 
\begin{equation}
 U_{\epsilon}(x_0)\leq \dfrac{\beta}{\beta -1} K_{\epsilon}.
\end{equation}

\end{proof}

\section{Adapted (heterogeneous) entropies}
\label{sec:entrop}

As explained in the introduction, entropies are useful to derive additional bounds and to characterize the limit as $\epsilon $ vanishes, see \cite{James,CLL, Tzavaras_relative}. We are going to use specific entropies adapted to  spatial dependence.

We recall the usual approach which is to define
\begin{definition}[Entropy pair]
 We  call an entropy pair for the system \eqref{a_evo2tubes} a couple of functions $(S, \, \Sigma)$ in BV$\big([0,1],\mathcal{C}(\mathbb{R}) \big)$ which satisfy
\begin{equation}
(i) \; S(.,x) \quad \mbox{and} \quad \Sigma(.,x) \;\mbox{are convex}, \qquad 
(ii) \; \dfrac{\partial S}{\partial u} (h(v,x),x) =\dfrac{\partial  \Sigma}{\partial v}(v,x).
\label{ad:entropy}
\end{equation}
\label{entropy}
\end{definition}

For such  entropy pairs, it is immediate to check that
\begin{equation*}
 \begin{aligned}
 \dfrac{\partial }{\partial t}[\;S(u_{\epsilon},x)&+ \Sigma(v_{\epsilon},x)  \; ]+ \dfrac{\partial}{\partial x}[\;S(u_{\epsilon},x)-\Sigma(v_{\epsilon},x)  \; ] \\
&=  \dfrac{1}{\epsilon}  \Big(\dfrac{\partial S}{\partial v}(u_{\epsilon},x) -\dfrac{\partial \S}{\partial v}(v_{\epsilon},x)  \Big) \Big( h(v_{\epsilon},x)-u_{\epsilon}\Big) 
 + \dfrac{\partial S}{\partial x}(u_{\epsilon},x)-\dfrac{\partial \S}{\partial x}(v_{\epsilon},x).\\
\end{aligned}
\end{equation*}
Now, using \eqref{ad:entropy}(ii), we can write
\begin{equation*}
 \begin{aligned}
 \dfrac{\partial }{\partial t}[ &S(u_{\epsilon},x)+ \Sigma(v_{\epsilon},x)  \; ] + \dfrac{\partial }{\partial  x}[\;S(u_{\epsilon},x)-\Sigma(v_{\epsilon},x)  \; ] 
 \\
&=  \dfrac{1}{\epsilon}  \Big(\dfrac{\partial S}{\partial v}(u_{\epsilon},x) -\dfrac{\partial S}{\partial v}(h(v_{\eps},x),x)  \Big) \Big( h(v_{\epsilon},x)-u_{\epsilon}\Big) 
 + \dfrac{\partial S}{\partial x}(u_{\epsilon},x)- \dfrac{\partial \S}{\partial x}(v_{\epsilon},x)
\\ 
& \leq  \dfrac{\partial S}{\partial x}(u_{\epsilon},x)- \dfrac{\partial \S}{\partial x}(v_{\epsilon},x),
\label{dissip}
\end{aligned}
\end{equation*}
because $S$ is convex with respect to its first variable and thus  $\dfrac{\partial S}{\partial v}$  is non decreasing with respect to its first variable.
\\

The shortcoming of Definition~\ref{entropy} is that the above  right hand side is not always well defined for $u_{\epsilon}$ and $v_{\epsilon}$ BV functions. Indeed, being given $S(u,x)$, we compute 
$$
 \Sigma(v,x)=\int_0^v S'_u\big(h( \bar v,x), x \big) d \bar v 
$$
and the expression for the $x$-derivative
$$
\dfrac{\partial  \Sigma }{\partial x}(v,x)  =  \int_0^v \left[ S''_{uu}\big(h( \bar v,x), x \big) h_x(\bar v,x)  + \cdots \right] d \bar v
$$
does not make intrinsic sense.

However, this entropy inequality is enough to prove that equilibrium is reached. Because the expressions $ \dfrac{\partial S}{\partial x}(u_{\epsilon},x)$ and $\dfrac{\partial \S}{\partial x}(v_{\epsilon},x)$ are bounded thanks to our assumptions on $h$, we may choose  
$$
S(u)=\dfrac{u^2}{2}, \qquad \Sigma(v,x)=\int_0^v h(\bar v,x)d\bar v,
$$
and using the entropy dissipation, with assumption \eqref{a_has}, we find after integration of equality \eqref{dissip} in ($x$,$t$), that for all $T>0$ for some constant $C(T)$ which does not depend on $\eps$, it holds 
$$
\frac 1 \epsilon \int_0^T \int_0^L \big|  u_{\epsilon}(x,t) -h(v_{\epsilon}(x,t) ,x) \big| dx dt \leq C(T).
$$
Therefore,  we arrive at the conclusion
\begin{proposition}
For $u_{\epsilon}$ and $v_{\epsilon}$ solutions of \eqref{a_evo2tubes}, we have the convergence
 \begin{equation}
 u_{\epsilon}-h(v_{\epsilon},x) \underset{\epsilon \longrightarrow 0}{\longrightarrow} 0 ,\qquad L^2\Big([0,L]\times[0,T]\Big).
 \end{equation}
\label{L2}
\end{proposition}
This convergence result toward equilibrium is an intermediate step in the proof of Proposition \ref{propo}.

We wish to go further and  avoid the two terms containing $x$-derivatives. For this goal, we define adapted  (heterogeneous) entropies by imposing more restrictions on the spatial dependence of the entropy pair.
\begin{definition}[An adapted  (heterogeneous) entropy family]
The pair of continuous functions  $(S,\Sigma)$, is an adapted  (heterogeneous) entropy for the system \eqref{a_evo2tubes}
if it satisfies the conditions of Definition~\ref{entropy} and if
\begin{equation}
 -\dfrac{\partial S}{\partial x}(h(v,x),x)+\dfrac{\partial \Sigma}{\partial x} (v,x)= 0, \qquad \forall v\geq 0. 
\label{homo_entropy}
\end{equation}
\label{def_homo_entropy}
\end{definition}

An example of such a family of entropies parametrized by $p\in  \mathbb{R}$ are the adapted Kru$\check{\mbox{z}}$kov entropies
 \begin{equation}
  S_p(u,x) = |u-h(k_p(x),x)|, \qquad \S_p(v,x) = |v-k_p(x)|, 
 \end{equation}
 where $(k_p)_{p \in \mathbb{R}}$ is the family of stationary solutions of the limit equation \eqref{CLlim}, which is equivalent to say 
\begin{equation*}
  h(k_{p}(x),x) -k_{p}(x)  = p.
\end{equation*} 
With this choice of an entropy pair, the above entropy inequality then reduces to 
$$
 \dfrac{\partial }{\partial t}[ S_p(u_{\epsilon},x)+ \Sigma_p(v_{\epsilon},x)  \; ] + \dfrac{\partial }{\partial  x}[\;S_p(u_{\epsilon},x)-\Sigma_p(v_{\epsilon},x)  \; ] \leq 0. 
$$

As a consequence, thanks to the strong compactness proven in next part, in the limit $ \eps \to 0$, the quasilinear conservation law \eqref {CLlim} comes with the family of adapted Kru$\check{\mbox{z}}$kov entropies. 
Indeed, if we define
\begin{equation}
\rho(x,t) : =h(v(x,t),x) + v(x,t), \qquad A(\rho,t): = h(v(x,t),x) - v(x,t), 
\label{cons}
\end{equation}
then, as in  \cite{AUD} and because $h$ is increasing, the following entropy inequality holds in the sense of distributions
\begin{equation}
 \dfrac{\p}{\p t}\Big| \rho(x,t)-\Big(k_p(x)+h(k_p(x),x)\Big)\Big| +  \dfrac{\p}{\p x}\Big| A(\rho,x)-A\Big(k_p(x)+h(k_p(x),x),x\Big)\Big| \leq 0.
 \label{limitt}
\end{equation}

\section{BV bounds}

For well prepared initial conditions \eqref{WP}, we present here our method to prove  BV bounds for appropriate quantities and strong compactness for $u_\epsilon$  and $v_\epsilon$, that are points (ii), (iii) of Proposition \ref{propo}.

\noindent{\bf 1st step. A bound on the time derivative at $t=0$.} 
Our first statement is 
 \begin{equation}
 \int_0^L | \dfrac{\partial u_{\epsilon}}{\partial t}(x,0)|dx \leq K_1(u^0), \qquad  \int_0^L | \dfrac{\partial v_{\epsilon}}{\partial t}(x,0)|dx \leq K_2(v^0).
\label{notice}
\end{equation}
Indeed, because initial conditions are at equilibrium, we have $ \dfrac{\partial v_{\epsilon}}{\partial t}(x,0) -  \dfrac{\partial v_{\epsilon}}{\partial x}(x,0)=0$.
We multiply this equality by $\sign\Big(\dfrac{\partial}{\partial t}v_{\epsilon}\Big)(x,0)$ and integrate over $[0,L]$ to get
\begin{equation}
 \int_0^L \Big| \dfrac{\partial v_{\epsilon}}{\partial t}(x,0)\Big|dx = \int_0^L \Big|\dfrac{\partial v_{\epsilon}}{\partial x}(x,0)\Big|dx 
  \leq K_2(v^0),
\end{equation}
and the above inequality follows from assumption \eqref{BV}. This  gives the first inequality of estimates \eqref{notice}. The same argument applies for $u_{\epsilon}$.

\noindent{\bf 2nd step. The time BV estimate.}
To prove \eqref{evolest2}, we differentiate each line of \eqref{a_evo2tubes} with respect to time and we multiply it respectively by $\sign\Big(\dfrac{\partial}{\partial t}u_{\epsilon} \Big)$ and $\sign\Big(\dfrac{\partial}{\partial t}v_{\epsilon} \Big)$ and integrate in $x$.
Adding the two lines, we obtain
\begin{equation}
\begin{aligned}
 \dfrac{d}{dt} \int_0^L [\; |\frac{\partial}{\partial t}u_{\epsilon}|+|\frac{\partial}{\partial t}v_{\epsilon}|\; ](x,t) dx
&\leq  
|\frac{\partial}{\partial t} u_0 |-| \frac{\partial}{\partial t} u_{\epsilon}(L,t)  |  +| \frac{\partial}{\partial t} v_{\epsilon}(L,t)  |-| \frac{\partial}{\partial t} v_{\epsilon}(0,t)  | \\
& = -\dfrac 12 | \frac{\partial}{\partial t} u_{\epsilon}(L,t)  | -| \frac{\partial}{\partial t} v_{\epsilon}(0,t)  | \leq 0,
\label{computation}
\end{aligned}
\end{equation}
which implies, using estimate \eqref{notice},
\begin{equation}
 \int_0^L [\; |\frac{\partial}{\partial t}u_{\epsilon}|+|\frac{\partial}{\partial t}v_{\epsilon}|\; ](x,t) dx \leq \int_0^L [\; |\frac{\partial}{\partial t}u_{\epsilon}|+|\frac{\partial}{\partial t}v_{\epsilon}|\; ](x,0) dx \leq K_1(u^0)+K_2(v^0).
 \label{evolest3}
\end{equation}

\noindent{\bf 3rd step. BV bound in $x$.}
 We complete the proof Proposition \ref{propo} (ii). Because of  space dependence of $h$ we cannot apply the same arguments for $x$-derivatives and build a single BV quantity. We add the two lines of \eqref{a_evo2tubes} and obtain
 
 \begin{equation*}
  \Big(  \dfrac{\p}{\p x}( u_{\epsilon} -v_{\epsilon} )\Big) (x,t)  =  -   \Big(  \dfrac{\p}{\p t}( u_{\epsilon} +v_{\epsilon}) \Big) (x,t) .
\end{equation*}
Using \eqref{evolest2}, we thus conclude that for all $t \geq 0$
 \begin{equation*}
\int_0^L \Big|\; \dfrac{\partial}{\partial x}(u_{\epsilon} -v_{\epsilon} )\; \Big|(x,t) dx \leq K_1(u^0)+K_2(v^0).
\end{equation*}

\noindent{\bf 4th step. Compactness.}

{Therefore we can conclude that $ \Big( u_{\epsilon} -v_{\epsilon}\Big) $ is compact in $L^1\Big([0,L]\times [0,T]\Big).$
Now, thanks to Proposition \eqref{L2},  $ \Big(h(v_{\epsilon},.) -u_{\epsilon}\Big) $ is compact in $L^1\Big([0,L]\times [0,T]\Big)$.
A combination of these two last compact embeddings gives us that
\begin{equation}
 h(v_{\epsilon},.) -v_{\epsilon} \qquad \mbox {is compact in }  L^1\Big([0,L]\times [0,T]\Big).
\end{equation}
Therefore, there is a function $Q\in L^\infty(0,L)$ such that, after extraction of a subsequence, 
\begin{equation*}
  h(v_{\epsilon},.) -v_{\epsilon}  \underset{\epsilon \longrightarrow 0}{\longrightarrow} Q(x),  
\end{equation*}
and because $v \mapsto h(v,x) -v$ is one-to-one, thanks to assumptions \eqref {as_h1}. In the same way, we conclude that the  sequence $ v_{\epsilon} $ converges. Gathering the informations above, Proposition~\ref{propo} (iii) is proved.

Then we can pass to the limit and obtain the last statement, Proposition~\ref{propo} (iv).

\section{Extension  to a more  specific relaxation system}

The method we have developed so far can be extended to a more realistic problem arising in kidney physiology that motivated this study.
The  system introduced and studied in \cite{TESP} is written, for $t\geq0$ and $x \in [0,L]$, 
\begin{equation}
\label{evo3tubes}
\left\{
\begin{aligned}
&\dfrac{\partial C_{\epsilon}^1}{\partial t}(x,t)+\dfrac{\partial C_{\epsilon}^1}{\partial x}(x,t)=\dfrac{1}{3\epsilon}\Big[ C_{\epsilon}^2(x,t)+h(C_{\epsilon}^3(x,t),x)) -2 C_{\epsilon}^1(x,t) \Big], 
\\
&\dfrac{\partial C_{\epsilon}^2}{\partial t}(x,t)+\dfrac{\partial C_{\epsilon}^2}{\partial x}(x,t)=\dfrac{1}{3\epsilon}\Big[ C_{\epsilon}^1(x,t)+h(C_{\epsilon}^3(x,t),x)) -2 C_{\epsilon}^2(x,t) \Big],
\\
&\dfrac{\partial C_{\epsilon}^3}{\partial t }(x,t)-\dfrac{\partial C_{\epsilon}^3}{\partial x}(x,t)=\dfrac{1}{3\epsilon}\Big[ C_{\epsilon}^1(x,t)+C_{\epsilon}^2(x,t)- 2 h(C_{\epsilon}^3(x,t),x))  \Big],
\\[3mm]
&C_{\epsilon}^1(0,t)=C_0^1, \qquad C_{\epsilon}^2(0,t)=C_0^2, \qquad C_{\epsilon}^3(L,t)=C_{\epsilon}^2(L,t),  \qquad t > 0, 
\end{aligned}
\right.
\end{equation}
Again, we want to prove uniform BV bounds for a small parameter $\epsilon$, which measures the ratio between ionic exchanges and flow along the tubules.

We make the same assumptions \eqref{as_h1} (the condition $1< \beta \leq \dfrac{\p h}{\p v}(v, x)$ can be relaxed to $1 \leq \dfrac{\p h}{\p v}(v, x)$)
and \eqref{a_has} and same hypotheses on the initial conditions, namely they belong to BV and are at equilibrium which means
$$
 C^1 =  C^2 = h(C^3,x ).
$$

 Following \eqref{cons}, the conservative quantity $\rho$ and the flux $B$ are defined by 
\begin{equation}
\rho(x,t) : = 2h(C^3(x,t),x) + C^3(x,t), \quad B(\rho,t): = 2h(C^3(x,t),x) - C^3(x,t) .
\end{equation}
For $p \in \mathbb{R}$, we define uniquely the steady state $k_p$ as 
 \begin{equation}
 B(k_p(x),x) = p.
  \label{k_p}
 \end{equation}

\begin{theorem}[Limit $\eps \rightarrow 0$] 
The functions $C_{\epsilon}^i$, $i=1, \; 2, \; 3$ converge almost everywhere to bounded functions $C^i$
and the quantity $\rho(x,t)$ is an entropy solution to
\begin{equation}
\left\{
\begin{aligned}
&\dfrac{\p }{\p t}\rho(x,t) +   \dfrac{\p}{\p x} B(\rho(x,t),x) =0, \qquad t>0, \; x \in[0,L], \\[3mm]
&\rho(0,t) = C_0^1 + C_0^2  + h^{-1}\big(\dfrac{C_0^1+C_0^2}{2},0\big), \qquad t>0, \\[3mm]
&\rho(x,0 ) = \rho^0(x), \qquad \rho^0(x):=C^0(x) + 2h(C^0(x),x),  \qquad x \in [0,L]. \\
\end{aligned}
\right.
\label{limit}
\end{equation}
\label{th2}
\end{theorem}
The entropy formulation of the conservation law, for adapted  (heterogeneous) entropies is written, following  \cite{AUD} again, 
$$
 \dfrac{\p}{\p t}\Big| \rho(x,t)-\Big(2 k_p(x)+h(k_p(x),x)\Big)\Big| +  \dfrac{\p}{\p x}\Big| B(\rho,x)-B\Big(2k_p(x)+h(k_p(x),x),x\Big)\Big| \leq 0.
$$
This family of inequalities is enough to prove uniqueness as in \cite{Nat_Ter} (see \cite{These_tournus} for details). 

The boundary conditions are understood in the following sense
\\
\textsf{Boundary condition at $x=0$}.  
For all $k_p$ such that $k_p(0) +2h(k_p(0),0) \in I\Big(\rho(0,t), h^{-1}\big(\dfrac{C_0^1+C_0^2}{2},0\big)+C_0^1+C_0^2\Big)$, we have 
 \begin{equation}
 \begin{aligned}
 \sign\Big( \rho(0,t) -h^{-1}(&\dfrac{C_0^1+C_0^2}{2},0)-C_0^1-C_0^2\Big) \\
 &\Big(B(\rho(0,t),0) - [2h(k_p(0),0)-k_p(0) ]\Big)
\leq 0,
 \end{aligned}
 \label{x=0}
 \end{equation}
\textsf{Boundary condition at $x=L$}. For all $k_p$ such that $k_p(L)+2h(k_p(L),L) \in I\Big(\rho(L,t), 2w_L + h^{-1}(w_L,L)\Big)$, we have 
 \begin{equation}
 \sign\Big( \rho(L,t) -2w_L-h^{-1}(w_L,L)\Big) \Big(B(\rho(L,t),L)-[2h(k_p(L),L)-k_p(L)] \Big) \geq 0,
 \label{x=L}
 \end{equation}
where $w_L(t):=\lim\limits_{\eps \longrightarrow 0} C_{\eps}^1(L,t)$, and where $I(a,b)$ denotes the interval $(\min(a,b),\max(a,b))$.
\\

Following the arguments we gave for the $2 \times 2$ system, and that we do not repeat, we can prove BV bounds in several steps
\\
(i)  $C_{\epsilon}^1$, $C_{\epsilon}^2$ and $C_{\epsilon}^3$ are bounded in $L^\infty((0,\infty)\times (0,L))$,
\\
(ii) $C_{\epsilon}^2(x,t)+h(C_{\epsilon}^3(x,t),x)) -2 C_{\epsilon}^1(x,t) \underset{ \epsilon \to 0 }{\longrightarrow}
 0$, $C_{\epsilon}^1(x,t)+h(C_{\epsilon}^3(x,t),x)) -2 C_{\epsilon}^2(x,t) \underset{ \epsilon \to 0 }{\longrightarrow}
 0$ in $L^2((0,\infty)\times (0,L))$,
\\
(iii) $\frac{\partial C_{\epsilon}^1}{\partial t}$, $\frac{\partial C_{\epsilon}^2}{\partial t}$, $\frac{\partial C_{\epsilon}^3}{\partial t}$ are bounded  in $L^\infty\big((0,\infty; L^1(0,L)\big)$,
\\
(iv) $\frac{\partial C_{\epsilon}^1}{\partial x}+ \frac{\partial C_{\epsilon}^2}{\partial x}+ \frac{\partial C_{\epsilon}^3}{\partial x}$ is bounded  in $L^\infty\big((0,\infty; L^1(0,L)\big)$.

These statements prove the convergence result in Theorem~\ref{th2}. 
\\

\noindent {\bf Acknowledgement.} 
Funding for this study was provided by the program EMERGENCE (EME 0918) of the Universit\'e Pierre et Marie Curie (Paris Univ. 6).


          %







          %
\bibliographystyle{siam}


%
%
%
%
%
%
\end{document}